\documentclass[12pt]{amsart}

\usepackage{latexsym,color,amsmath,amsthm,amssymb,amscd,amsfonts,dsfont,graphicx,comment}

\usepackage{multirow}
\usepackage{arydshln}
\usepackage{verbatim}

\usepackage{fullpage}

\newtheorem{theorem}{Theorem}[section]
\newtheorem{lemma}[theorem]{Lemma}
\newtheorem{prop}[theorem]{Proposition}
\theoremstyle{definition}
\newtheorem{defi}[theorem]{Definition}
\newtheorem{rem}[theorem]{Remark}

\newcommand{\R}{\mathbb{R}}
\newcommand{\Rn}{\mathbb{R}^n}
\newcommand{\B}{\mathbb{B}^n}
\newcommand{\emp}{\emptyset}
\newcommand{\N}{\mathbb{N}}
\newcommand{\ti}{\times}

\newcommand{\diag}{{\rm{diag}}}

\newcommand{\Ri}{\rightarrow}
\newcommand{\Sub}{\subseteq}
\newcommand{\Sup}{\supseteq}
\newcommand{\ifff}{\Leftrightarrow}

\newcommand{\rh}{\rho}
\newcommand{\la}{\lambda}
\newcommand{\ka}{\kappa}
\newcommand{\ez}{\epsilon}
\newcommand{\T}{\Theta}
\newcommand{\wt}{\widetilde}
\newcommand{\sz}{\sigma}

\def\r{\right}
\def\lf{\left}
\def\ls{\lesssim}
\def\fz{\infty}

\begin{document}

\title{A characterization of spaces of homogeneous type induced by continuous ellipsoid covers of $\R^n$}

\author{Marcin Bownik}
\address{Department of Mathematics, University of Oregon, Eugene, OR 97403--1222, USA}
\email{mbownik@uoregon.edu}

\author{Baode Li}
\address{College of Mathematics and System Science, Xinjiang University, Urumqi, 830046, P. R. China}
\email{baodeli@xju.edu.cn}

\author{Tal Weissblat}
\address{Einstein Institute of Mathematics,
Hebrew University of Jerusalem,
Givat Ram, 9190401, Israel}
\email{tal.weissblat@mail.huji.ac.il}



\thanks{Marcin Bownik was partially supported by NSF grant DMS-1956395.
Baode Li was supported by the National Natural Science Foundation of China (Grant No.
11861062).}
\date{\today}

\begin{abstract}
We study the relationship between the concept of a continuous ellipsoid $\Theta$ cover of $\mathbb{R}^n$, which was introduced by Dahmen, Dekel, and Petrushev \cite{DDP0, DDP, DP}, and the space of homogeneous type induced by $\Theta$. We characterize the class of quasi-distances on $\R^n$ (up to equivalence) which correspond to continuous ellipsoid covers. This places firmly continuous ellipsoid covers as a subclass of spaces of homogeneous type on $\R^n$ satisfying quasi-convexity and $1$-Ahlfors-regularity.
\end{abstract}

\maketitle

\section{Introduction}\label{s1}

Discrete and continuous ellipsoid covers of $\R^n$ were introduced by Dahmen, Dekel, and Petrushev in the construction and analysis of multilevel preconditioners for partition of unity methods applied to elliptic boundary value problems \cite{DDP0} and in the study of Besov spaces with pointwise variable anisotropy \cite{DDP, De}, see also the survey \cite{DP}. A  continuous ellipsoid cover consists of ellipsoids $\theta_{x,t}$ with centers $x\in\R^n$ and scales $t\in \R$ satisfying a natural shape condition. Dekel, Han, and Petrushev \cite{DHP} have shown that an ellipsoid cover $\Theta$ defines a space of homogeneous type in the sense of Coifman and Weiss \cite{cw71, cw77} with a quasi-distance $\rho_\Theta$ given by
\begin{equation}\label{ed}
\rho_\Theta(x,y):=\inf_{\theta\in\Theta}\lf\{|\theta|:x,y\in\theta\r\}.
\end{equation}
More precisely, $\R^n$ equipped with the Lebesgue measure and quasi-distance $\rho_\Theta$ is $1$-Ahlfors regular, i.e., Lebesgue measure of balls satisfy $|B_{\rho_\Theta}(x,r)| \sim r$ for all $x\in\R^n$ and $r>0$. Subsequently, Dekel, Petrushev, and Weissblat \cite{DPW} have developed the Hardy spaces $H^p(\Theta)$ associated with a continuous ellipsoid cover $\Theta$ for the entire range of $0<p \le 1$. Among the results shown in this setting are grand maximal function characterization, atomic decomposition, and classification of Hardy spaces \cite{DPW}, the duality of Hardy spaces \cite{dw}, molecular decomposition \cite{abr}, and boundedness of Calder\'on-Zygmund singular integral operators \cite{bll}. In contrast with the general theory of Hardy spaces on spaces of homogenous type \cite{am, cw77, hs}, these results work in the full range $0<p\leq 1$. This is actually the largest class of spaces of homogeneous type on $\R^n$ equipped with Lebesgue measure, where such complete $H^p$ theory has been developed so far.

A natural question arises about the relationship between ellipsoid covers and spaces of homogeneous type on $\R^n$. What quasi-distances on $\R^n$ are induced by continuous ellipsoid covers? In this paper we answer this question by characterizing all quasi-distances (up to equivalence) which correspond to continuous ellipsoid covers via the formula \eqref{ed}. In addition that $\rho$ is $1$-Ahlfors regular, we impose that $\rho$ is quasi-convex. That is, there exists a constant $Q\ge 1$ such that for every $x\in \R^n$ and $r>0$ there exists an ellipsoid $\xi=\xi_x^r$ with center $x$ such that
\begin{equation}\label{de}
\xi_x^r\Sub B_{\rho}(x,r) \Sub Q \cdot \xi_x^r,
\end{equation}
where $Q \cdot \xi= Q(\xi-x)+x$ is a dilate of an ellipsoid $\xi$ by a factor $Q$. The famous  maximal volume ellipsoid theorem of John \cite{ball, FJ, TV} attests that every convex body in $\R^n$ is $Q$-quasi-convex with $Q=n$. Hence, the above definition is a natural generalization of convexity reminiscent of the concept of a quasi-conformal mapping \cite{H}.

The main result of the paper shows that there is one-to-one correspondence between quasi-convex, $1$-Ahlfors-regular quasi-distances and continuous ellipsoid covers in $\R^n$. In this correspondence we identify equivalent quasi-distances and likewise equivalent ellipsoid covers. In other words, a quasi-convex, $1$-Ahlfors-regular quasi-distance $\rho$ gives rise to a continuous ellipsoid cover $\Xi=\{\xi_x^r: x\in \R^n, r>0\}$, where $\xi_x^r$ satisfies \eqref{de}. In turn, a quasi-distance $\rho_\Xi$, which is induced by $\Xi$ and given by \eqref{ed}, is  quasi-convex and $1$-Ahlfors-regular, and $\rho_\Xi$ is equivalent to $\rho$.

While the methods of the proof are quite elementary and require mostly basic properties of ellipsoids, some of them could not be found in the existing literature such as Theorem \ref{t2.x1}. The most demanding arguments revolve around the inner property which guarantees appropriate growth of balls $B_\rho(x,r)$ as $r\to \infty$. It turns out that this property is automatically implied by the quasi-convexity and $1$-Ahlfors-regularity of $\rho$. In turn, the inner property plays a key role in showing that $\Xi=\{\xi_x^r: x\in \R^n, r>0\}$ satisfies the shape condition, which is the key requirement for $\Xi$ to be a continuous ellipsoid cover.

This article is organized as follows. Section \ref{s2} is devoted to proving basic properties of ellipsoids such as Theorem \ref{t2.x1}. In Section \ref{s3} we introduce the notion of a continuous ellipsoid cover, recall some of its known properties and prove new ones. In Section \ref{s4} we study quasi-convexity and the inner property and show the main characterization result of the paper, Theorem 4.9. Finally, in Section \ref{s5} we give applications and examples of quasi-distances illustrating our main result.

\section{Ellipsoids in $\R^n$}\label{s2}

In this section we recall some basic properties of ellipsoids in $\R^n$.
An {\it ellipsoid} $\xi$ in $\Rn$ is an image of the closed Euclidean unit ball $\B$ in $\Rn$ under an affine map, i.e.,
$$\xi =M_\xi(\B)+c_\xi,$$ where $M_\xi$ is an $n\times n$ nonsingular matrix and $c_\xi\in\Rn$ is the center of ellipsoid $\xi$. For any ellipsoid $\xi$ and $\la>0$, define a dilated ellipsoid by $$\la \cdot \xi:=\la M_\xi(\B)+c_\xi.$$

The following elementary theorem shows that if one ellipsoid is contained in the other, then we have a reverse inclusion relation for a dilated ellipsoid.

\begin{theorem}\label{t2.x1}
If two ellipsoids $\eta$ and $\xi$ in $\Rn$ satisfy $\eta\Sub\xi$, then
$$\xi\Sub 2\frac{|\xi|}{|\eta|} \cdot \eta.$$
Moreover, if $\eta$ and $\xi$ have the same center, then the above holds without the factor $2$.
\end{theorem}

Since we could not find Theorem \ref{t2.x1} in the literature, we will give its proof using three more elementary lemmas.

\begin{lemma} \label{l1.1}
Let $D:=\diag(\la_1, \la_2,\dots, \la_n)$ be a diagonal matrix. If $\B\Sub D(\B)+c$
with $c\in\Rn$, then $\B\Sub D(\B)$.
\end{lemma}

\begin{proof}
We only need to verify  $|\la_i|\ge 1$ for any $i= 1,\dots, n$.
Let $e_1,\ldots,e_n$ be a standard basis of $\R^n$.
Note that
\[
D(\B)\Sub \{ x=(x_1,\ldots, x_n) \in \R^n: |x_i| \le |\lambda_i| \quad\text{for all }i=1,\ldots,n\}.
\]
If $\B -c \Sub D(\B)$, then the absolute value of $i$'th coordinate of $e_i - c$ or $-e_i-c$ is $\ge 1$. Hence, by the above inclusion we have $ |\lambda_i|\ge 1$ for every $i=1,\ldots,n$.
\end{proof}

\begin{lemma} \label{l1.2}
Let $A$ be a nonsigular matrix. Let $\{\la_i\}_{i=1}^n$ be the eigenvalues of $AA^T$ and $D:=\diag(\sqrt{\la_1}, \sqrt{\la_2},\dots, \sqrt{\la_n})$.
If  $\B\Sub A(\B)+c$ with
$c\in\Rn$, then  there exists an orthogonal matrix $U$ such that $\B\Sub UA(\B)=D(\B)$. In particular, $\lambda_i \ge 1$ for all $i=1,\ldots,n$.
\end{lemma}

\begin{proof}
Since $AA^T$ is a positive
symmetric matrix, then there exists  an orthogonal matrix $U$ such that $UAA^TU^T=UA(UA)^T=D^2$. Therefore,
\begin{align*}
D(\B)&=\{Dx\in\R^n: x^Tx\le 1\}=\{x\in\R^n: x^T(D^2)^{-1}x\le 1\}\\
&=\{x\in\R^n: x^T(UA(UA)^T)^{-1}x\le 1\}=UA(\B).
\end{align*}
Since $\B\Sub A(\B)+c$, we have
$$\B=U(\B)\Sub UA(\B)+Uc=D(\B)+Uc.$$
Hence, by Lemma \ref{l1.2}, we have $\B\Sub UA(\B)=D(\B)$.
\end{proof}

\begin{lemma} \label{Lemma for help 1}
If two ellipsoids satisfy $\eta\Sub \xi$, then  $\eta-c_\eta\Sub \xi-c_\xi$, where $c_\eta$ and $c_\xi$ are centers of $\eta$ and $\xi$, resp.
\end{lemma}

\begin{proof}
Without loss of generality, we can assume that $c_\eta= 0$ by using translations. Let
$$\eta:=M_\eta(\B),\qquad \xi:=M_\xi(\B)+c_\xi,$$
for some nonsingular matrices $M_\eta$ and $M_\xi$.
Since
$$\B =(M_\eta)^{-1}\eta \Sub  (M_\eta)^{-1}\xi= (M_\eta)^{-1}M_\xi \B + (M_\eta)^{-1}c_\xi,$$
by Lemma \ref{l1.2} we have $\B  \Sub (M_\eta)^{-1}M_\xi \B$, which yields the required conclusion.
\end{proof}

\begin{proof}[Proof of Theorem \ref{t2.x1}]
Take any two ellipsoids $\eta:=M_\eta(\B)+c_\eta\Sub\xi:=M_\xi(\B)+c_\xi$. Without loss of generality, we may assume that $c_\eta= 0$ by using translations.
Let $A=(M_\eta)^{-1}M_\xi$. By Lemma \ref{Lemma for help 1}, we
have
\begin{align}\label{e2.x3}
\B = (M_\eta)^{-1}(\eta)\Sub (M_\eta)^{-1}(\xi-c_\xi) = A\B.
\end{align}
Let $D:=\diag(\sqrt{\la_1}, \sqrt{\la_2},\dots, \sqrt{\la_n})$,
where $\{\la_i\}_{i=1}^n$ are the eigenvalues of $AA^T$.
By Lemma \ref{l1.2}, there exits an orthogonal matrix $U$ such that $UA(\B)=D(\B)$ and hence
\begin{equation}\label{e2.x1}
 U^{-1}D(\B)=(M_\eta)^{-1}(\xi-c_\xi).
\end{equation}
Since $|\det U|=1$ and $\la_i\ge1$ for all $i= 1,\ldots, n$, we have
\begin{equation}\label{e1.1}
\frac{|\xi|}{|\eta|}=\frac{|(M_\eta)^{-1}(\xi-c_\xi)|}{|(M_\eta)^{-1}(\eta)|}
=\frac{|U^{-1}(D(\B))|}{|\B|}
=\prod_{i=1}^n\sqrt{\la_i}\ge \max_{1\le i\le n}\sqrt{\la_i}.
\end{equation}
Therefore, by \eqref{e2.x1} and \eqref{e1.1} we obtain
\begin{equation}\label{e1.2}
(M_\eta)^{-1}(\xi-c_\xi)=U^{-1}D(\B)\Sub U^{-1}\max_{1\le i\le n}\sqrt{\la_i} \B=\max_{1\le i\le n}\sqrt{\la_i} \B\Sub \frac{|\xi|}{|\eta|}(M_\eta)^{-1}(\eta).
\end{equation}

Moreover, using the assumption $\eta\Sub \xi$, we get
$$\B=(M_\eta)^{-1}(\eta)\Sub (M_\eta)^{-1}(\xi)=(M_\eta)^{-1}(M_\xi(\B)+c_\xi).$$
By this and $UA (\B)=D(\B)$, we have
$$ \B=U(\B)\Sub U(M_\eta)^{-1}(M_\xi(\B)+c_\xi)=D(\B) +U(M_\eta)^{-1} c_\xi.$$
This implies that
$\B-U(M_\eta)^{-1} c_\xi\Sub D(\B)$
and hence
$$U(M_\eta)^{-1} c_\xi\in -D(\B)\Sub -\max_{1\le i\le n}\sqrt{\la_i} \B=\max_{1\le i\le n}\sqrt{\la_i} \B.$$
Combining this with \eqref{e1.1} yields
$$ (M_\eta)^{-1} c_\xi\in \max_{1\le i\le n}\sqrt{\la_i} \B \Sub \frac{|\xi|}{|\eta|}(M_\eta)^{-1}(\eta). $$
Hence, by \eqref{e1.2} we have
\begin{equation*}
(M_\eta)^{-1}(\xi)\Sub\frac{|\xi|}{|\eta|}(M_\eta)^{-1}(\eta)+(M_\eta)^{-1}c_\xi\Sub 2\frac{|\xi|}{|\eta|}(M_\eta)^{-1}(\eta).
\end{equation*}
Applying  $M_\eta$ to both sides we finally
obtain $\xi\Sub  2\frac{|\xi|}{|\eta|} \cdot \eta$.

Finally, if $\eta\Sub\xi$ have the same center, then we may assume that $c_\eta=c_\xi=0$. Hence, \eqref{e1.2} alone implies that $\xi\Sub \frac{|\xi|}{|\eta|}\eta$.\end{proof}


\section{Ellipsoid covers and quasi-distances on $\R^n$}\label{s3}

In this section we recall the properties of a continuous ellipsoid cover $\Theta$, which was originally introduced by Dahmen, Dekel, and Petrushev \cite{DDP}. This includes properties of quasi-distance $\rho_\Theta$ which is induced by the cover $\Theta$. Moreover, we translate the shape condition of $\Theta$ into a geometric form involving only containment of dilates of ellipsoids in $\Theta$.

\begin{defi}\label{d3.x1}
We say that
$$\T:=\{ \theta_{x,t}: x\in\Rn,t\in\R\}$$
is a {\it continuous ellipsoid cover } of $\R^n$, or shortly a  cover,
if there exist positive constants ${\mathbf p}(\T):=\{a_1,\ldots, a_6\}$ such that:
\begin{itemize}
\item[(i)]
For every $x\in \R^n$ and $t\in \R$, there exists an ellipsoid $\theta_{x,t}:=M_{x,t}(\B)+x$, where $M_{x,t}$ is a real $n\times n$ nonsingular matrix, satisfying
\begin{eqnarray}\label{volum}
a_12^{-t}\leq|\theta_{x,t}|\leq a_2 2^{-t}.
\end{eqnarray}
\item[(ii)]
Intersecting ellipsoids in $\T$ satisfy the {\it shape condition} requiring that for any $x,y\in \Rn$, $t\in \R$ and $s\ge0$, if $\theta_{x,t}\cap \theta_{y,t+s}\ne \emp $, then
\begin{equation}
\label{e3.y6}
a_3 2^{-a_4 s}\le 1 / \| (M_{y,t+s})^{-1} M_{x,t}\|
\le \|(M_{x,t})^{-1} M_{y,t+s}\|
\le a_5 2^{-a_6 s}.
\end{equation}
Here, $\|\cdot \|$ is the matrix norm given by $\|A\|:=\max_{|x|=1}|Ax|$ for any nonsingular matrix $A$.
\end{itemize}
\end{defi}

It is worth emphasizing that we do not assume any measurability or continuity condition on a continuous ellipsoid cover $\Theta$. Indeed, by \cite[Theorem 2.2]{bll} there exists an equivalent ellipsoid cover such that its corresponding matrix valued function $x \mapsto M_{x,t}$ is continuous for any $t\in \R$.

\begin{rem}\label{r3.y1}
The  shape condition (ii) in Definition \ref{d3.x1} has the following equivalent formulation by reversing scales. For any $x,y\in \Rn$, $t\in \R$ and $s\ge0$, if $\theta_{x,t}\cap \theta_{y,t-s}\ne \emp $, then
    \begin{eqnarray}\label{e3.w1}
    \frac1{a_5} 2^{a_6 s}\le 1 / \| (M_{y,t-s})^{-1} M_{x,t}\|
\le \|(M_{x,t})^{-1} M_{y,t-s}\|
\le \frac1{a_3} 2^{a_4 s}.
    \end{eqnarray}
Indeed, \eqref{e3.w1} follows from \eqref{e3.y6} applied to  $\theta_{y,t-s}$ and $\theta_{x,t}$ in place of $\theta_{x,t}$ and  $\theta_{y,t+s} $, resp. Reversing this argument, shows the converse implication.
\end{rem}

The  shape condition \eqref{e3.y6} can be also equivalently restated in terms of  dilates of the ellipsoids in $\T$ without referring to scale parameter $t$.

\begin{lemma}\label{geo}
Let $\T=\{ \theta_{x,t}: x\in\Rn,t\in\R\}$ be a collection of ellipsoids satisfying \eqref{volum}. Then, the shape condition \eqref{e3.y6} holds if and only if there exists constants $a_3'$ and $a_5'$ such that
for any two ellipsoids $\xi$, $\eta\in\T$, if $|\eta|\le |\xi|$ and $\xi\cap\eta\ne\emptyset$, then
\begin{equation}\label{e3.y1}
a_3'\bigg(\frac{|\eta|}{|\xi|}\bigg)^{a_4} (\xi-c_\xi) \Sub \eta-c_\eta
\Sub a_5' \bigg(\frac{|\eta|}{|\xi|} \bigg)^{a_6} (\xi-c_\xi),
\end{equation}
where  $c_\xi$ and $c_\eta$ are the centers of $\xi$ and $\eta$, resp.
\end{lemma}

\begin{proof}
By \eqref{volum} for any $t,s\in \R$ we have
\begin{equation}\label{v7}
\frac{a_1}{a_2}2^{-s} \le \frac{|\theta_{y, t+s}|}{|\theta_{x, t}|} \le \frac{a_2}{a_1}2^{-s}
\end{equation}
Hence, if $|\theta_{y, t+s}| \le |\theta_{x, t}|$, then $s\ge -s_0$, where $s_0:=\log_2(a_2/a_1) \ge 0$. As a partial converse, if $s\ge s_0$, then $|\theta_{y, t+s}| \le |\theta_{x, t}|$.

Suppose that the shape condition \eqref{e3.y6} holds for $t\in \R$, $s\ge 0$, and  $\theta_{x,t}\cap \theta_{y,t+s}\ne \emp $. First, we shall show that the same condition also holds for $s\ge -s_0$, albeit for some new constants $a_3'$ and $a_5'$. Indeed, if $s\ge 0$, then there is nothing new to show. Suppose next $-s_0 \le s \le 0$. Then, by the reverse form of \eqref{e3.y6}, see Remark \ref{r3.y1}, we have
\[
\frac{1}{a_5} 2^{-a_6 s} \le
1/\|(M_{y,t+s})^{-1} M_{x,t}\| \le
\| (M_{x,t})^{-1} M_{y,t+s}\|
\le \frac{1}{a_3} 2^{-a_4 s}.
\]
Take $\tilde a_3= \min(a_3,1/a_5 )$ and $\tilde a_5=\max(a_5,(1/a_3) 2^{(a_6-a_4)s_0})$. Since $a_6 \le a_4$, for $-s_0 \le s \le 0$ we have
\[
\tilde a_3 2^{-a_4 s} \le \frac{1}{a_5} 2^{-a_6 s} \quad\text{and}\quad  \frac{1}{a_3} 2^{-a_4 s}\le  \tilde a_5  2^{-a_6 s}.
\]
Therefore, if $t\in \R$, $s\ge -s_0$, and  $\theta_{x,t}\cap \theta_{y,t+s}\ne \emp $, then
\begin{equation}
\label{e3.y7}
\tilde a_3 2^{-a_4 s}\le 1 / \| (M_{y,t+s})^{-1} M_{x,t}\|
\le \|(M_{x,t})^{-1} M_{y,t+s}\|
\le \tilde a_5 2^{-a_6 s}.
\end{equation}

Now suppose we have two ellipsoids $\xi,\eta\in\T$ such that $|\eta|\le |\xi|$ and $\xi\cap\eta\ne\emptyset$. We write $\eta =\theta_{y, t+s}$ and $\xi = \theta_{x, t}$ for some $x,y \in \R^n$ and $t,s \in \R$. Since $|\eta|\le |\xi|$, we necessarily have  $s\ge -s_0$. By the right-hand side inequality of \eqref{e3.y7} we have
$
(M_{x,t})^{-1} M_{y,t+s}(\B) \Sub \tilde a_5 2^{-a_6 s}\B
$.
Hence, \eqref{v7} implies that
\begin{equation}\label{in1}
 M_{y,t+s}(\B) \Sub a_5' \lf(\frac{|\theta_{y,t+s}|}{|\theta_{x,t}|}\r)^{a_6 }M_{x,t}(\B),
\end{equation}
where $a_5'= \tilde a_5( a_2/a_1)^{a_6}$. Applying the same argument for the left-hand side inequality of \eqref{e3.y7} yields
\begin{equation}\label{in2}
 a_3' \lf(\frac{|\theta_{y,t+s}|}{|\theta_{x,t}|}\r)^{a_4 }M_{x,t}(\B)
 \Sub  M_{y,t+s}(\B) ,
\end{equation}
where $a_3'=\tilde a_3 (a_1/a_2)^{a_4}$. This shows \eqref{e3.y1}.

Conversely, suppose that \eqref{e3.y1} holds for $\xi,\eta\in\T$, $|\eta|\le |\xi|$, and $\xi\cap\eta\ne\emptyset$. We claim that the same condition holds when $|\eta|\le (a_2/a_1) |\xi|$, albeit for some new constants $\check a_3$ and $\check a_5$. Indeed, if $|\eta|\le |\xi|$, then there is nothing new to show. Suppose that $|\xi| \le |\eta| \le (a_2/a_1)|\xi|$. Then, by \eqref{e3.y1} and by reversing order of inclusions we have
\[
\frac{1}{a_5'}\bigg(\frac{|\eta|}{|\xi|}\bigg)^{a_6} (\xi-c_\xi) \Sub \eta-c_\eta
\Sub \frac1{a_3'} \bigg(\frac{|\eta|}{|\xi|} \bigg)^{a_4} (\xi-c_\xi).
\]
Hence,  \eqref{e3.y1} holds with constants $\check a_3 = \min(a'_3,1/a_5')$ and $\check a_5 = \max(a'_5,(1/a_3)(a_2/a_1)^{a_4-a_6} )$ in place of $a_3'$ and $a_5'$, resp. Now, take any $x,y \in\R^n$, $t\in \R$, and $s\ge 0$ such that $\theta_{x,t}\cap \theta_{y,t+s}\ne \emp $. Letting $\eta=\theta_{y,t+s}$ and $\xi=\theta_{x,t}$, \eqref{e3.y1} yields \eqref{in1} and \eqref{in2}. Converting these inclusions into norm inequalities using \eqref{v7} yields \eqref{e3.y6} for appropriate constants $a_3$ and $a_5$ .
\end{proof}

\begin{rem}
As a consequence of Lemma \ref{geo} we propose the alternative geometric definition of an ellipsoid cover $\Theta$, which will be used in our consideration in Section \ref{s4}. A collection $\T=\{ \xi_{x}^r: x\in\Rn,r>0\}$ is a continuous ellipsoid cover if there exist positive constants ${\mathbf p}(\T):=\{a_1,\ldots, a_6\}$
such that:
\begin{itemize}
\item[(i)]
for every $x\in \R^n$ and $r>0$, $\xi_x^r$ is an ellipsoid with center $x$ and volume satisfying
\[
a_1 r\leq|\xi_{x}^r|\leq a_2 r,
\]
\item[(ii)]
for any ellipsoids $\xi$, $\eta\in\T$, such that $|\eta|\le |\xi|$ and $\xi\cap\eta\ne\emptyset$, we have \eqref{e3.y1}.
\end{itemize}
To translate between two formulations involving scale $t\in \R$ and ``radius'' $r>0$, it suffices to take $\theta_{x,\, t} = \xi_x^{r}$, where $r=2^{-t}$, and then apply Lemma \ref{geo}.
\end{rem}


The following lemma from \cite[Lemma 2.2]{DPW} is a direct consequence of the shape condition \eqref{e3.y6}.

\begin{lemma}\label{l4.3}
Let $\T$ be a continuous ellipsoid cover. Then there exists $c>0$ depending only on ${\mathbf p}(\T)$  such that for any $x\in\Rn,\,t\in\R$ and $\lambda\ge1$, we have $\lambda \cdot \theta_{x,\,t}\Sub \theta_{x,\,t-c\lambda}$.
\end{lemma}

The following lemma is a continuous analogue of \cite[Lemma 2.8]{DDP}, which was originally shown for discrete ellipsoid covers. Hence, for the sake of completeness we include its proof.

\begin{lemma}\label{l4.1}
Let $\T$ be a continuous ellipsoid cover. Then there exists a constant $s^\ast\ge 0$ depending only on ${\mathbf p}(\T)$ such that for any ellipsoids $\theta_{x,\,t}$ and $\theta_{y,\,t+s}$ with $\theta_{x,\,t}\cap \theta_{y,\,t+s}\ne \emptyset$, where $x,\, y\in\Rn,\, t\in\R$ and $s\ge 0$, we have
$\theta_{x,\,t}\cup \theta_{y,\,t+s}\Sub \theta_{x,\,t-\ell}$ for any $\ell\ge s^*$,.
\end{lemma}

\begin{proof}
We write $\theta_{x,\,t}:=M_{x,t}(\B)+x$, $\theta_{y,\,t+s}:=M_{y,\,t+s}(\B)+y$, and let
$\omega:=(M_{x,\,t})^{-1}(\theta_{y,t+s}-x)$. Then by the shape condition \eqref{e3.y6} and $s\ge0$, we have
\begin{align*}
\mathrm{diam}(\omega)&:=\sup_{z,\,z'\in \omega}|z-z'|=\sup_{z,\,z'\in \B}
|(M_{x,\,t})^{-1}M_{y,\,t+s}(z-z')|\\
&\le2\|(M_{x,\,t})^{-1}M_{y,\,t+s}\|\le 2a_52^{-a_6s}
\le 2a_5.
\end{align*}
This, together with $\theta_{x,\,t}\cap \theta_{y,\,t+s}\ne \emptyset$, implies that
\begin{align}\label{e4.w2}
(M_{x,\,t})^{-1}[(\theta_{x,\,t}\cup\theta_{y,\,t+s})-x]=\B\cup \omega \Sub  (1+2a_5)\B .
\end{align}
Therefore, we have
\[
\theta_{x,\,t}\cup\theta_{y,\,t+s} \Sub (1+2a_5)  \cdot \theta_{x,\, t}
\]
On the other hand, by Lemma \ref{l4.3} we have for any $\lambda \ge 1+2a_5$,
\[
(1+2a_5)  \cdot \theta_{x,\, t} \Sub \lambda  \cdot \theta_{x,\, t} \Sub \theta_{x,\,t-\lambda c}.
\]
Hence, Lemma \ref{l4.1} holds for $s^*=(1+a_4)c$.
\end{proof}

Next we move to exploring the relationship between continuous ellipsoid covers and quasi-distances on $\R^n$.

\begin{defi} \label{d3.y1}
A mapping $\rh:\Rn\ti\Rn \Ri [0,\infty)$ is called a {\it quasi-distance} if there exists a positive  constant $\ka\ge1$ such that for all $x,y,z\in\Rn$,
\begin{itemize}
\item[(i)]  $\rh(x,y)=0 \ifff x=y$;
\item[(ii)]  $\rh(x,y)=\rh(y,x) $;
\item[(iii)] $\rh(x,z) \leq \ka(\rh(x,y)+\rh(y,z))$.
\end{itemize}
\end{defi}

Dahmen, Dekel, and Petrushev have shown that an ellipsoid cover $\Theta$ induces a quasi-distance $\rho_\Theta$ on $\R^n$, see \cite[Proposition 2.7]{DDP}. Moreover, $\R^n$ equipped with the quasi-distance $\rho_\Theta$ and the Lebesgue measure is a space of homogeneous type which is Ahlfors $1$-regular \cite[Proposition 2.10]{DDP}. These results can be summarized as follows.

\begin{prop}\label{p3.x2}
Let $\Theta$ be a continuous ellipsoid cover. The function $\rho_\Theta:\R^n\times \R^n\rightarrow [0,\fz)$ defined by
\begin{equation}\label{e3.10}
\rho_\Theta(x,y):=\inf_{\theta\in\Theta}\lf\{|\theta|:x,y\in\theta\r\}
\end{equation}
is a quasi-distance on $\R^n$. Moreover, the Lebesgue measure of balls
\begin{equation}\label{e3.3}
B_{\rho_\Theta}(x,r)=\{y\in\R^n:\rho_\Theta(x,y)<r\}
\end{equation}
with respect to the quasi-distance $\rho_\Theta$ satisfies
\begin{equation}\label{ahlfors}
|B_{\rho_\Theta}(x,r)| \sim r \qquad\text{for all }x\in \R^n, \ r>0,
\end{equation}
with equivalence constants depending only on $\mathbf{p}(\Theta)$.
\end{prop}

The condition \eqref{ahlfors} states the Lebesgue measure is $1$-Ahlfors regular with respect the quasi-distance $\rho_\Theta$. This immediately implies the doubling property $|B_{\rho_\Theta}(x,2r)| \ls |B_{\rho_\Theta}(x,r)|$, which is a defining feature of spaces of homogeneous type introduced by Coifman and Weiss \cite{cw71, cw77}.

 The following result is stated without the proof in \cite[Theorem 2.7]{DPW}. Its proof can be found in \cite[Proposition 2.10]{bll}.

\begin{prop}\label{p4.x1}
Let $\Theta$ be a continuous ellipsoid cover and let $\rho_\Theta$ be a quasi-distance as in \eqref{e3.10}. For any ball $B_{\rho_\Theta}(x,r)$ with $x\in\R^n$ and $r>0$, there exist $t_1, t_2\in \R$ such that
\[
\theta_{x, \,t_1}\subset B_{\rho_\Theta}(x,r)
\subset\theta_{x,\,t_2}
\qquad\text{and}\qquad
|\theta_{x, \,t_1}|\sim |\theta_{x,\,t_2}|\sim r,
\]
where equivalence constants depend only on $\mathbf{p}(\Theta)$.
\end{prop}

Using Proposition \ref{p4.x1} we can introduce a more convenient variant of a quasi-distance induced by a continuous ellipsoid cover.

\begin{prop}\label{p3.x4}
Let $\T$ be an ellipsoid cover. For any $x,\,y\in\Rn$, define
\[
\rh_1(x,\,y):=\inf_{y\in\theta_{x,\,t}\in\T}|\theta_{x,\,t}|
\qquad\text{and}\qquad
\rh_2(x,\,y):=\inf_{x\in\theta_{y,\,t}\in\T}|\theta_{y,\,t}|.
\]
Then the map $\wt\rh_\T(x,\,y):=\min\{\rh_1(x,\,y),\,\rh_2(x,\,y)\}$ is a quasi-distance which is equivalent to $\rh_\T(x,\,y)$ as in \eqref{e3.10}.
\end{prop}

\begin{proof}
It suffices to show that
 \begin{align}\label{e3.s1}
\rh_\T(x,\,y)\sim \rh_1(x,\,y)\quad \mbox{for\ any \ } x,\,y\in\Rn.
\end{align}
 Indeed, if \eqref{e3.s1} holds, then by symmetry we have
 $\rh_\T(x,\,y)\sim\rh_2(x,\,y)$, and therefore
 $$\rh_\T(x,\,y)\sim\min\{\rh_1(x,\,y),\,\rh_2(x,\,y)\}=\wt\rh_\T(x,\,y).$$
Since $\wt\rh_\T(x,\,y)=\wt\rh_\T(y,\,x)$, the fact that $\rh_\T$ is
 a quasi-distance (see Proposition \ref{p3.x2}), implies that $\wt\rh_\T$ is also a
 quasi-distance which is equivalent to $\rh_\T$.

Obviously, $\rh_\T(x,\,y)\le \rh_1(x,\,y)$, so it remains to prove that
there exists a constant $C>0$ such that $\rh_1(x,\,y)\le C\rh_\T(x,\,y)$.
Let $r:=\rh_\T(x,\,y)$. By Proposition \ref{p4.x1}, there exist two ellipsoids $\theta_{x,\,t_1},\,\theta_{x,\,t_2}$ with $|\theta_{x,\,t_1}|\sim|\theta_{x,\,t_2}|\sim r$ such that
$$\theta_{x,\,t_1} \Sub B_{\rh_\T}(x,\,2r)\Sub  \theta_{x,\,t_2}.  $$
Since $y\in B_{\rh_\T}(x,\,r)$, by the definition of $\rh_1(x,\,y)$, it follows that
$$\rh_1(x,\,y)\le |\theta_{x,\,t_2}|\sim r=\rh_\T(x,\,y),$$
which completes the proof of Proposition \ref{p3.x4}.
\end{proof}

\section{Quasi-convex quasi-distances on $\R^n$}\label{s4}

In this section we show that the quasi-distance $\rho_\Theta$, induced by a continuous ellipsoid cover $\Theta$, is not only $1$-Ahlfors-regular, but it also satisfies two other crucial properties: quasi-convexity and the inner property. We also show the converse statement that any quasi-convex, $1$-Ahlfors-regular quasi-distance $\rho$ automatically satisfies the inner property and generates a continuous ellipsoid cover $\Xi$. In addition, the quasi-distance $\rho_\Xi$, induced by $\Xi$, is equivalent to $\rho$. This constitutes the main result of the paper.

We start by recalling properties of convex bodies in $\R^n$. A {\it convex body} in $\Rn$ is a compact convex set with nonempty interior.
Fritz John \cite[p.\,202, Theorem III]{FJ} proved that every convex body in $\R^n$ contains a unique ellipsoid of maximal volume. The dilate by the dimension $n$ of such ellipsoid contains the original convex body, see \cite{ball} and \cite[Theorem 3.13]{TV}.

\begin{theorem}\label{Thm. John}
Let $K\Sub\Rn$ be a convex body. Then there exists a unique ellipsoid $\xi \Sub \R^n$ of maximal volume such that $\xi \Sub K$. Moreover, $K\Sub n \cdot \xi$.
\end{theorem}

Motivated by Theorem $\ref{Thm. John}$ and the concept of quasiconformal mapping \cite{H} we can generalize the notion of convexity.

\begin{defi}
Let $Q \ge1$. We say that a subset $K'\Sub\Rn$ is $Q$-{\it quasi-convex} with respect to $x\in K'$, if there exists an ellipsoid $\xi\Sub\Rn$ with center $c_\xi=x$ such that
\begin{equation}\label{qc}
\xi \Sub K' \Sub Q \cdot \xi.
\end{equation}
\end{defi}

By Theorem \ref{Thm. John}, any convex body in $\R^n$ is $Q$-quasi-convex with respect the center of the unique maximal volume ellipsoid contained in the convex body, where $Q=n$. Notice that we do not impose uniqueness in the above definition. Namely, for a given set $K'$ there could be two different ellipsoids (even of maximal volume) satisfying \eqref{qc}.

\begin{defi}\label{quasi convex}
Given a quasi-distance $\rh:\Rn\ti\Rn\Ri[0,\infty)$, we say that $\rh$ is {\it quasi-convex} if there exists $Q\ge1$ such that for any $x\in\Rn$ and $r>0$, the ball
\[
B_\rh(x,r):=\{y\in\Rn:\rh(x,y)<r \}
\]
is $Q$-quasi-convex with respect to $x$. That is, there exists an ellipsoid $\xi^r_x$ with center $x$ such that
\begin{equation}\label{qce}
\xi^r_x\Sub B_\rh(x,r) \Sub Q \cdot \xi^r_x.
\end{equation}
In this case we define the corresponding family of ellipsoids
\begin{equation}\label{qc0}
\Xi_{\rh}:=\{\xi^r_x: x\in\Rn, r>0  \}.
\end{equation}
\end{defi}

\begin{lemma}\label{p4.1}
For any continuous ellipsoid cover  $\T$, the induced quasi-distance $\rh_\T$ given by \eqref{e3.10} is quasi-convex.
\end{lemma}

\begin{proof} For any ball $B_{\rh_\T}(x,r)$, by Proposition \ref{p4.x1},  there exist two ellipsoids $\theta_{x, t_1},\theta_{x, t_2}\in\T$ and two constants $d_2\ge d_1>0$, which depend only on ${\mathbf p}(\T)$, such that   $\theta_{x, t_1}\Sub B_{\rh_\T}(x,r)\Sub \theta_{x, t_2}$  and
$$d_1r \leq |\theta_{x, t_1}|\leq |B_{\rh_\T}(x,r)|\leq |\theta_{x, t_2}| \le d_2 r.$$
Since $\theta_{x, t_1}\Sub\theta_{x, t_2}$, by Theorem \ref{t2.x1} we conclude that $\theta_{x, t_2}\Sub \frac{|\theta_{x, t_2}|}{|\theta_{x, t_1}|}\theta_{x, t_1}\Sub \frac{d_2}{d_1}\theta_{x, t_1}$. Therefore, we have
$$
\theta_{x, t_1}\Sub B_{\rh_\T}(x,r)\Sub \frac{d_2}{d_1}\theta_{x, t_1}.
$$
This proves that the induced quasi-distance $\rh_\T$ is quasi-convex with $Q=d_2/d_1$.
\end{proof}

We introduce yet another property of a quasi-distance which will play an important role in our considerations.

\begin{defi} \label{d3.yy1}
We say that a quasi-distance $\rh$ on $\R^n$ satisfies the {\it inner property} if there exist  constants $a= a(\rho), b=b(\rho)>0$ such that for any $x\in\Rn,\,r>0$ and $\la \ge 1$,
\begin{equation} \label{ball inner property}
a\la^b(B_\rho(x,\,r)-x)\Sub B_\rh(x,\,\la  r)-x.
\end{equation}
\end{defi}

The inner property is stronger than the reverse doubling property \cite{YZ} since it immediately implies that
\[
a^n \lambda^{bn} |B_\rho(x,\,r)| \le |B_\rh(x,\,\la  r)| \qquad\text{for all }\lambda\ge 1.
\]
While the inner property \eqref{ball inner property} of $\rh$ is formulated in terms of balls, it can also be equivalently phrased in terms of ellipsoids in $\Xi_\rho$.

\begin{lemma}\label{p3.x3}
Let $\rh$ be a quasi-distance on $\R^n$, which is quasi-convex. Let $\xi^r_x$ be the corresponding ellipsoids as in Definition \ref{quasi convex}. Then $\rho$ satisfies the inner property if and only if  there exist  positive constants $a_1, b_1>0$ such that for any $x\in\Rn,\,r>0$ and $\la \ge 1$,
\begin{equation} \label{inner property}
a_1\la^{b_1} \cdot \xi^r_x\Sub \xi^{\la r}_x.
\end{equation}
\end{lemma}

\begin{proof}
Since $\rh$ is quasi convex, for every $x\in \R^n$ and $r>0$, there exists an ellipsoid $\xi_x^r$ such that
\eqref{qce} holds. By \eqref{ball inner property}, it
follows that, for any $x\in\Rn$, $r>0$ and $\la\ge1$,
\begin{eqnarray*}
a\la^b(\xi^r_x-x)\Sub a\la^b(B_\rh(x,\,r)-x)\Sub B_\rh(x,\,\la r)-x\Sub Q(\xi^{\la r}_x-x).
\end{eqnarray*}
Hence, \eqref{inner property} holds true with $a_1=a/Q$ and $b_1=b$. Similarly we can show that \eqref{inner property} implies  \eqref{ball inner property} with $a=a_1/Q$ and $b=b_1$.
\end{proof}

The following lemma implies that intersecting ellipsoids in $\Xi_\rh$ of comparable volume  have similar shapes.

\begin{lemma} \label{p3.x1}
Let $\rh$ be a quasi-distance which is quasi-convex and $1$-Ahlfors-regular. That is, there exists a constant $c_1\ge 1$ such that
\begin{equation}\label{vol}
\frac{1}{c_1} r \le |B_\rh(x,r)| \le c_1 r \qquad \mathrm{for \ all \ }x\in \R^n, r>0.
\end{equation}
Let $\Xi_\rho$ be a family of ellipsoids corresponding to $\rho$ as in Definition \ref{quasi convex} and let  $c_2\ge 1$.
Suppose that $\xi=\xi^r_x$, $\eta=\xi^s_y\in \Xi_\rho$, $x,y\in\R^n$, $r,s>0$, are such that
\begin{equation}\label{p2}
B_\rho(x,r) \cap \eta\neq \emp \qquad\text{and}\qquad
 |\eta| \le c_2|\xi|.
\end{equation}
Then there exists a constant $c  \ge 1$, which depends only on $c_1$, $c_2$, the triangle inequality constant $\kappa$, and the quasi-convexity parameter $Q$, such that
$\eta \Sub c \cdot \xi$.
\end{lemma}

\begin{proof}
By the quasi-convexity of $\rho$, there exists  $Q\ge 1$ such that
\begin{equation}\label{p1}
\xi\Sub B_\rh(x, r)\Sub Q \cdot \xi,\qquad \eta\Sub B_\rh(y, s)\Sub Q \cdot \eta.
\end{equation}
Hence, by \eqref{vol} we have
\begin{equation}\label{qce2}
\frac{1}{Q^n c_1} r \le |\xi| \le c_1 r \qquad\text{and}\qquad
\frac{1}{Q^n c_1} s \le |\eta| \le c_1 s.
\end{equation}
Combining  \eqref{p2}, \eqref{p1}, and \eqref{qce2} we have
$$
 s\le c_3 r, \qquad \text{where } c_3:= Q^n(c_1)^2c_2.
$$
Since  $B_\rh(x,r)\cap B_\rh(y,s)\ne \emp$, the triangle inequality of $\rh$, and the quasi-convex property of $\rh$, implies that
$$\eta\Sub B_\rh(y, s)\Sub  B_\rh(x, \ka (r+2\kappa c_3r)) \Sub B_\rh(x, 3\ka^2c_3r)\Sub Q\cdot \xi^{3\ka^2c_3r}_x. $$
By \eqref{qce2} we have
\[
\frac{ |Q \cdot \xi^{3\ka^2c_3r}_x|}{|\xi|}
= \frac{ |Q|^n | \xi^{3\ka^2c_3r}_x|}{|\xi|}
\le c:=Q^{2n}(c_1)^2 3\ka^2c_3.
\]
Since $\xi$ and $Q \cdot \xi^{3\ka^2c_3r}_x$ have the same center, Theorem \ref{t2.x1} yields
$$
\eta\Sub Q \cdot \xi_{x}^{3\ka^2c_3r}\Sub \frac{ |Q \cdot \xi^{3\ka^2c_3r}_x|}{|\xi|}\cdot \xi \Sub c \cdot \xi,$$
which completes the proof of Lemma \ref{p3.x1}.
\end{proof}

Next we show that the inner property holds automatically for quasi-convex and $1$-Ahlfors-regular quasi-distances.

\begin{theorem}\label{auto}
Let $\rho$ be a quasi-distance on $\R^n$  which is quasi-convex and $1$-Ahlfors-regular. Then $\rho$ satisfies the inner property.
\end{theorem}

\begin{proof}
First, we will show that there exists $d=d(\rh)>1$  such that
\begin{equation}\label{auto0}
d (B_\rho(x,\,r)-x)\Sub B_\rh(x,\,2\kappa  r)-x.
\end{equation}
Indeed, since $\rho$ is $Q$-quasi-convex, then for any $x\in\R^n$ and $r>0$ there exists an ellipsoid $\xi^r_x\in\Xi_\rho$ such that
$$\xi^r_x-x\Sub B_\rho(x,r)-x\Sub Q(\xi^r_x-x),$$
and for any $y\in B_\rho(x,r)$ there exists an ellipsoid $\xi^r_y\in\Xi_\rho$ such that $$\xi^r_y-y\Sub B_\rho(y,r)-y\Sub Q(\xi^r_y-y).$$
By \eqref{qce2} we have
\[
\frac{1}{(c_1)^2 Q} \le \frac{|\xi^r_x|}{|\xi^r_y|} \le (c_1)^2 Q.
\]
Since $B_\rho(x,r) \cap \xi^r_y \neq\emptyset$, by Lemma \ref{p3.x1}, there exists a positive constant $c$ such that
$\xi^r_y \Sub c \cdot \xi^r_x$. By Lemma \ref{Lemma for help 1} we have $\xi^r_y -y\Sub c  (\xi^r_x-x)$. Hence, by Theorem \ref{t2.x1}
we have
\begin{equation}\label{e4.x3}
c  (\xi^r_x-x) \Sub \frac{|c  \xi^r_x|}{|\xi^r_y|}( \xi^r_y -y) = c^n  (c_1)^2 Q ( \xi^r_y -y).
\end{equation}

Let $d>1$ be such that $(d-1)c^{n-1} (c_1)^2Q^2 = 1$.
Take any $z\in d(B_\rho(x,r)-x)+x$. Let $y\in B_\rho(x,r)$ be such that
$$z=d(y-x)+x=y+(d-1)(y-x).$$
By \eqref{e4.x3} and our choice of $d$ we have
\begin{equation*}
(d-1)Q(\xi^r_x-x)\Sub (d-1)Q c^{n-1} (c_1)^2Q (\xi^r_y-y)\Sub B_\rho(y,r)-y.
\end{equation*}
Since $y-x\in Q(\xi^r_x-x)$, we further deduce that
$$z=y+(d-1)(y-x)\in y+(B_\rho(y,r)-y)=B_\rho(y,r).$$
By the triangle inequality
$$\rho(z,\,x)\le\kappa(\rho(x,\,y)+\rho(y,\,z))\le 2\kappa r.$$
This implies that $z\in B_\rho(x,\,2\kappa r)$ and hence \eqref{auto0} holds.

Now we can verify the inner property of $\rho$. Take $\ez>0$ such that $d=(2\kappa)^\ez$. For any $\la\ge1$, there exists $\ell\in\N_0$ and
$$(2\kappa)^\ell\le \la<(2\kappa)^{\ell+1}.$$
Hence, by \eqref{auto0} it follows that
\begin{align*}
B_\rh(x,\,\la r)-x&\Sup B_\rh(x,\, (2\kappa)^\ell r)-x\Sup d^\ell(B_\rh(x,\,r)-x)
\Sup d^{-1}\la^\ez(B_\rh(x,\,r)-x).
\end{align*}
Therefore, the inner property \eqref{ball inner property} holds with $a=d^{-1}$ and $b=\ez$.
\end{proof}

The main result of the paper shows a $1$-to-$1$ correspondence between equivalence classes of continuous  ellipsoid covers  and quasi-convex, $1$-Ahlfors-regular quasi-distances on $\R^n$.

\begin{theorem}\label{t4.1}
(i) For any  continuous  ellipsoid cover  $\T$, the induced
quasi-distance $\rh_\T$ given by \eqref{e3.10} is quasi-convex and $1$-Ahlfors-regular.

(ii) Conversely, for any quasi-convex and $1$-Ahlfors-regular quasi-distance $\rho$  on $\R^n$, the corresponding family $\Xi=\Xi_\rho$, given by Definition \ref{quasi convex}, is a continuous ellipsoid cover. Moreover, its induced quasi-distance
\begin{equation}\label{rht}
\rh_\Xi(x,\,y):=\inf_{\xi\in\Xi}\{|\xi|:\, x,\,y\in\xi\},\qquad x,y\in\R^n,
\end{equation}
is equivalent to the original quasi-distance $\rh$.
\end{theorem}

\begin{proof}
Part (i) follows by Proposition \ref{p3.x2} and Lemma \ref{p4.1}. Moreover, by Theorem \ref{auto} we can deduce that $\rho_\T$ has the inner property.

To prove (ii), we first verify that a family $\Xi_\rh$ induced by quasi-distance $\rh$ is a continuous ellipsoid cover. By \eqref{qce2}, there exists a constant $c_1>0$
such that for any $x\in\Rn$ and $r>0$,
\begin{equation}\label{qce4}
\frac1{Q^nc_1}r\le |\xi^r_x|\le c_1 r.
\end{equation}
Therefore, by letting $\theta_{x,\,t}:=\xi^r_x$ with $t=-\log_2r$, we obtain \eqref{volum} for $a_1:=\frac1{Q^nc_1}$ and $a_2:=c_1$.

To show that $\Xi_\rh$ satisfies the shape condition \eqref{e3.y6}, by Lemma \ref{geo} it suffices to verify \eqref{e3.y1}. Consider two ellipsoids $\xi=\xi^r_x$, $\eta=\xi^s_y$ in $\Xi_\rh$, where  $x,\,y\in\Rn$, $r, s>0$, such that  $\xi\cap\eta \neq \emp$ and $|\eta| \le |\xi|$. By Lemma \ref{p3.x1}, there exists a constant $c>1$ such that $\eta\Sub c \cdot \xi$. Hence,
by Lemma \ref{Lemma for help 1}, we have
\begin{equation}\label{qce6}
\eta-y\Sub c\cdot \xi-x = c(\xi -x).
\end{equation}
Applying Theorem \ref{t2.x1} yields
$$
c(\xi-x)\Sub \frac{|c \xi|}{|\eta|}(\eta-y)=c^{n}\frac{|\xi|}{|\eta|}(\eta-y).
$$
Thus,
$$
c^{1-n} \frac{|\eta|}{|\xi|} (\xi-x) \Sub \eta-y.
$$
This shows the left-hand side inclusion of \eqref{e3.y1} with $a_3':=c^{1-n}$ and $a'_4:=1$.


Next we show the right-hand side inclusion of \eqref{e3.y1}. By Theorem \ref{auto} quasi-distance $\rho$ satisfies the inner property. Hence, by Lemma \ref{p3.x3} there exists positive constants $a_1$ and $b_1$ such that \eqref{inner property} holds. Note that we necessarily have $a_1\le 1$ by letting $\lambda=1$. Assume first that
\begin{equation}\label{qce7}
\frac{r}{s} > {a_1}^{-1/{b_1}} \ge 1.
\end{equation}
Then the inner property \eqref{inner property} for $\lambda=r/s$, implies
\begin{equation}\label{qce5}
\xi^s_y \Sub a_1 \lambda^{b_1} \cdot \xi^s_y \Sub \xi^{\lambda s}_y =\xi^r_y.
\end{equation}
Hence, $\xi^r_x \cap \xi^r_y \ne \emptyset$. Moreover, by \eqref{qce4}
\[
|\xi^r_y| \le c_1 r \le (c_1)^2 Q^n |\xi^r_x|.
\]
Hence, by Lemma \ref{p3.x1} applied for $c_2=(c_1)^2 Q^n $, there exists a constant $c'$ such that  that $\xi^r_y \Sub c' \cdot \xi^r_x$.  Combining this with \eqref{qce5} and Lemma \ref{Lemma for help 1} we have
\begin{equation}\label{qce11}
a_1 \bigg(\frac{r}{s} \bigg) ^{b_1} (\xi^s_y-y) \Sub c'(\xi^r_x -x).
\end{equation}
On other hand, by \eqref{qce4} we have
\begin{equation}\label{qce8}
\frac{|\xi^r_x|}{|\xi^s_y|}\le (c_1)^2 Q^n \frac{r}s.
\end{equation}
Therefore, remembering that $\xi=\xi^r_x$ and $\eta=\xi^s_y$, \eqref{qce11} and \eqref{qce8} imply that
\[
\eta - y \Sub  ((c_1)^2 Q^n )^{b_1} \frac{c'}{a_1} \bigg( \frac{|\eta|}{|\xi|} \bigg)^{b_1} (\xi -x ).
\]
This shows the left-hand side inclusion of \eqref{e3.y1} with $a_5':=((c_1)^2 Q^n )^{b_1} \frac{c'}{a_1} $ and $a_6:=b_1$ under the assumption \eqref{qce7}.

Next assume that ${r}/{s} \le  {a_1}^{-1/{b_1}}$. Then, by \eqref{qce8}
\[
 \bigg( \frac{|\eta|}{|\xi|} \bigg)^{b_1} \ge a_1 ((c_1)^2 Q^n)^{-b_1}.
\]
Combing this with \eqref{qce6} implies
\[
\eta- y \Sub  ((c_1)^2 Q^n )^{b_1} \frac{c}{a_1} \bigg( \frac{|\eta|}{|\xi|} \bigg)^{b_1} (\xi -x ).
\]
Again we have deduced the left-hand side inclusion of \eqref{e3.y1} albeit with $a_5':=((c_1)^2 Q^n )^{b_1} \frac{c}{a_1} $. By Lemma \ref{geo} we conclude  that  $\Xi_\rh$ is a continuous ellipsoid cover.

Finally we prove the equivalence of $\rho$ and $\rho_\Xi$.  For every $x\in \R^n$ and $t\in \R$ we set $\tilde \theta_{x,t} = \xi_x^r$, where $r=2^{-t}$. We have just shown that
\[
\Xi_\rho=\{\tilde \theta_{x,t}: x\in \R^n, t\in\R\}
\]
is a continuous ellipsoid cover.

Take any $x \ne y\in\Rn$. Let $r=\rho(x,\,y)/2$ and $t=-\log_2 r$. By the quasi-convex property of $\rh$, there exists a constant $Q\ge1$ such that
$$\tilde \theta_{x,\,t} = \xi^r_x \Sub B_\rh(x,\,r)\Sub Q\cdot \xi^r_x = Q \cdot \tilde \theta_{x,\,t}.$$
By Lemma \ref{l4.3}, there exists a constant $c>0$ such that
$$x,\, y\in B_\rho(x,\,r)\Sub Q\cdot \tilde \theta_{x,\,t}\Sub \tilde \theta_{x,\, t-cQ}.$$
By \eqref{volum}, \eqref{rht}, and $2^{-t}=r=\rh(x,\,y)/2$,
it follows that
\begin{align}\label{e4.x1}
\rh_\Xi(x, y)\le |\tilde \theta_{x,\,t-c Q}|\le a_2 2^{c Q}r=a_2 2^{c Q-1}\rho(x, y).
\end{align}

On the other hand, by the definition of $\rh_\Xi$, there exists an ellipsoid
$\xi^{\wt r}_z\in\Xi_\rho$, $z\in \R^n$, $\tilde r>0$, such that $x,\,y\in\xi^{\wt r}_z$ and $|\xi^{\wt r}_z|\le 2\rh_\Xi(x,\,y)$. Moreover, by the quasi-convexity of $\rho$,
$$\xi^{\wt r}_z\Sub B_\rho(z,\, \wt r)\Sub Q\cdot \xi^{\wt r}_z.$$
Since $x,\, y\in B_\rho(z,\, \wt r)$ and  $\wt r\le c_1|B_\rh(z,\,\wt r)|$ ($\rho$ is $1$-Ahlfors-regular) we have
\begin{equation}\label{e4.x2}
\rho(x, y)\le \kappa [\rho(x, z)+\rho(z, y)]\le 2c_1\kappa |B_\rho(z,\, \wt r)|\le 2c_1\kappa|Q\cdot \xi^{\wt r}_z|\le 4c_1\kappa Q^n\rh_\Xi(x, y).
\end{equation}
Combining \eqref{e4.x1} with \eqref{e4.x2} yields equivalence of quasi-distances $\rho$ and $\rho_\Xi$.
\end{proof}

\section{Applications and examples}\label{s5}

In this section we give applications and examples of quasi-distances illustrating our main result, Theorem \ref{t4.1}.
As a consequence of results about Hardy $H^p(\Theta)$ spaces  with variable anisotropy associated with continuous ellipsoid cover $\Theta$, which were introduced by Dekel, Petrushev, and Weissblat in \cite{DPW}, we deduce the following result.

\begin{theorem}\label{hardy} Suppose that  $\rho$ is a quasi-convex and $1$-Ahlfors-regular quasi-distance on $\R^n$. Then, $\R^n$ equipped with $\rho$ and the Lebesgue measure is a space of homogeneous type for which Hardy space $H^p(\R^n,\rho)$ spaces exists for the entire range $0<p \le 1$. These spaces admit grand maximal function characterization, atomic decomposition, molecular decomposition, and their duals are Campanato spaces. Moreover, there exists a class of Calder\'on-Zygmund singular integral operators which are bounded on $H^p(\R^n,\rho)$ spaces for $0<p \le 1$.
\end{theorem}

To wit Theorem \ref{hardy} we define $H^p(\R^n,\rho)$ as the anisotropic Hardy space $H^p(\Xi_\rho)$,  where $\Xi_\rho$ is a continuous ellipsoid cover corresponding to quasi-distance $\rho$ as in Theorem \ref{t4.1}. Consequently, $H^p(\R^n,\rho)=H^p(\Xi_\rho)$ enjoys all properties of Hardy spaces with variable anisotropy shown in \cite{abr, bll, DPW, dw}.

Our first example involves a family of ellipses $\T_0:=\{\theta_{x,\,t}:\ x\in\R^2,\, t\in\R\}$ with
$$\theta_{x,\,t}:= \lf\{z=(z_1,\,z_2)\in\R^2:\ \frac{(z_1-x_1)^2}{\sz^2_1}
+\frac{(z_2-x_2)^2}{\sz^2_2}\le 1\r\}$$
where semi-axes $\sigma_1$ and $\sigma_2$ are given by the following table:

\begin{table}[h]
  \centering
  \linespread{2.5}
 \resizebox{0.5\hsize}{!}{
    \begin{tabular}{|c|c|c|c|}
    \hline
    $t$  & $x_2$ & $\sz_1$ & $\sz_2$ \\
    \hline
    $t\le 0$ & $\R$ & $2^{-\frac t2}$ & $2^{-\frac t2}$ \\
    \hline
    $t>0$ & $|x_2|>2^{-\frac t3} $ & $2^{-\frac t2}$ & $2^{-\frac{t}2}$ \\
    \hline
      $t>0$ & $2^{-\frac t2}<|x_2|\le 2^{-\frac t3}$ & $2^{-\frac{5t}6}\frac1{|x_2|}$  & $2^{-\frac t6}|x_2|$ \\
    \hline
    $t>0$ & $|x_2|\le 2^{-\frac t2}$& $2^{-\frac t3}$ & $2^{-\frac{2t}3}$ \\
    \hline
    \end{tabular}
  }
\end{table}
We will show that $\T_0$ is a continuous ellipsoid cover and give the formula of quasi-norm $\rh_{\T_0}$ induced by $\T_0$. For this we need an elementary lemma.

\begin{lemma}\label{l3.ss1}
Let $a_i,\,\beta_i>0$, $i=1,\,2,\,\dots,\, d$, where $d\ge 2$. Then the root $x>0$ of the equation
$\sum^d_{i=1} a_ix^{\beta_i}=1$, satisfies $x\sim b:=\min_{1\le i\le d} a_i^{-1/\beta_i}$. More precisely,
$$ \min_{1\le i\le d} d^{-1/\beta_i} b<x
  <b. $$
\end{lemma}
\begin{proof}
For $y>0$ define $f(y)=\sum_{1\le i\le d}a_iy^{\beta_i}$. It is easy to see that
$f$ is strictly increasing and $f(x)=1<f(b)$, which implies that $x<b$. Since
$$ f \bigg (\min_{1\le j\le d} d^{-1/\beta_j}b \bigg )\le \frac1d\sum^d_{i=1} a_i \min_{1\le j\le d}\frac1{a_j}<1=f(x),$$
we deduce that $\min_{1\le i\le d}d^{-1/\beta_i}b<x$.
\end{proof}

\begin{prop}\label{p3.z1}
$\T_0$ with  is a continuous ellipsoid cover in the sense of Definition
\ref{d3.x1}.
\end{prop}
\begin{proof}
It is obvious that $\T_0$ satisfies Definition \ref{d3.x1}(i). We only need to show that any two intersecting ellipsoids $\theta_{x,\,t},\,\theta_{y,\,t+s}\in\T_0$  satisfy Definition \ref{d3.x1}(ii), where $t\in\R$ and $s\ge0$. We shall verify two typical cases while other cases are similar or trivial. Denote by $\sz_2$ the vertical semi-axis of $\theta_{x,\,t}$ and
by $\sz'_2$ the vertical semi-axis of $\theta_{y,\,t+s}$.

{\bf Case 1.} Suppose that  $2^{-t/2}<|x_2|\le 2^{-t/3}$
and $|y_2|\le 2^{-(t+s)/2}$, where $t>0$ and $s\ge 0$. Then, we have
$$M_{x,\,t}=\diag(2^{-5t/6}/|x_2|,\, 2^{-t/6}|x_2|),\ M_{y,\,t+s}=\diag(2^{-(t+s)/3},\,2^{-2(t+s)/3}).$$
By $\theta_{x,\,t}\cap\theta_{y,\,t+s}\not=\emptyset$, $|y_2|\le 2^{-(t+s)/2}$,
$\sz_2=2^{-t/6}|x_2|$, $\sz'_2=2^{-2(t+s)/3}$, $|x_2|\le 2^{-t/3}$, $t>0$ and $s\ge0$,
we know that
$$|x_2|\le |y_2|+\sz_2+\sz'_2\le  2^{-(t+s)/2}+2^{-t/2}+2^{-2(t+s)/3}\le 3\cdot2^{-t/2}.$$
From this and $|x_2|>2^{-t/2}$, it follows that
\[
\|(M_{x,\,t})^{-1}M_{y,\,t+s}\|=\|\diag(2^{t/2-s/3}|x_2|,\, 2^{-t/2-2s/3}/|x_2|)\|\le
\|\diag(3\cdot2^{-s/3},\, 2^{-2s/3})\|\le 3\cdot 2^{-s/3}
\]
and
\begin{align*}
\|(M_{y,\,t+s})^{-1}M_{x,\,t}\|=\|\diag(2^{-t/2+s/3}/|x_2|,\, 2^{t/2+2s/3}|x_2|)\|\le
\|\diag(2^{s/3},\,3\cdot2^{2s/3} \|\le 3\cdot 2^{2s/3}.
\end{align*}

{\bf Case 2.} Suppose that $|x_2|\le 2^{-t/2}$
and $2^{-(t+s)/2}<|y_2|\le 2^{-(t+s)/3}$, where $t>0$ and $s\ge 0$. Then, we have
$$ M_{x,\,t}=\diag(2^{-t/3},\,2^{-2t/3}),\ M_{y,\,t+s}=\diag(2^{-5(t+s)/6}/|y_2|,\, 2^{-(t+s)/6}|y_2|).$$
By $\theta_{x,\,t}\cap\theta_{y,\,t+s}\not=\emptyset$, $|x_2|\le 2^{-t/2}$,
$\sz_2=2^{-2t/3}$, $\sz'_2=2^{-(t+s)/6}|y_2|$, $|y_2|\le 2^{-(t+s)/3}$, $t>0$ and $s\ge0$,
we know that
$$|y_2|\le |x_2|+\sz_2+\sz'_2\le  2^{-t/2}+2^{-2t/3}+2^{-(t+s)/2}\le 3\cdot2^{-t/2}.$$
From this and $|y_2|>2^{-(t+s)/2}$, it follows that
\begin{align*}
\|(M_{x,\,t})^{-1}M_{y,\,t+s}\|=\|\diag(2^{-t/2-5s/6}/|y_2|,\, 2^{t/2-s/6}|y_2|)\|\le
\|\diag(2^{-s/3},\, 3\cdot2^{-s/6})\|\le 3\cdot 2^{-s/6}
\end{align*}
and
\begin{align*}
\|(M_{y,\,t+s})^{-1}M_{x,\,t}\|=\|\diag(2^{t/2+5s/6}|y_2|,\, 2^{-t/2+s/6}/|y_2|)\|\le
\|\diag(3\cdot2^{5s/6},\, 2^{2s/3})\|\le 3\cdot 2^{5s/6}.
\end{align*}
\end{proof}

\begin{prop}\label{p3.z2}
The quasi-distance $\rho_{\Theta_0}$ induced by the ellipsoid cover $\T_0$ satisfies
$$
\rh_{\T_0}(x,\,y)\sim
\lf\{
\begin{array}{ll}
 |x-y|^2 &   |x-y|\ge 1 \ \mathrm{or} \ |x-y|^{\frac23}<|x_2|,\\
\lf[(y_1-x_1)^2+\sqrt{(y_1-x_1)^4+4(y_2-x_2)^2}\r]^{\frac34} & |x-y|<1\ \&\ |x_2|\le\varphi(x,\,y), \\
\max\{(x_1-y_1)^{\frac65}|x_2|^{\frac65} ,\,(x_2-y_2)^6|x_2|^{-6}\} &|x-y|<1\\
 &\&\, \varphi(x,\,y)<|x_2|\le |x-y|^{2/3},
\end{array}
\r.
$$
where $\varphi(x,\,y):=2^{-\frac34}\lf[(y_1-x_1)^2+\sqrt{(y_1-x_1)^4+4(y_2-x_2)^2}\r]^{\frac34}.$
\end{prop}
\begin{proof}
By Proposition \ref{p3.x4}, it suffice to find the formula for
\[
\rh_1(x,\,y):=\inf_{y\in\theta_{x,\,t}\in\T}|\theta_{x,\,t}|.
\]
It is not hard to verify that ellipses in $\T_0$ are nested, i.e., $\theta_{x,\,t_1}\subsetneq \theta_{x,\,t_2}$ for any $x\in\R^2$ and $t_1,\,t_2\in\R$ with $t_1>t_2$. Using this and the fact that ellipses in $\T_0$ are closed, we know that $\rh_{1}(x,\,y)$ equals to the area of an ellipse $\theta_{x,\,t}$ for some $t\in\R$ such that $y$ belongs to the boundary of $\theta_{x,\,t}$, i.e.,
$y\in\partial\theta_{x,\,t}$. Equivalently,
\begin{eqnarray}\label{e3.q1}
\frac{(y_1-x_1)^2}{\sz^2_1}
+\frac{(y_2-x_2)^2}{\sz^2_2}=1.
\end{eqnarray}
We shall consider three cases.

{\bf Case 1.} Suppose that $t\in\R$ and $\theta_{x,\,t}$ is a ball. Since $y\in\partial\theta_{x,\,t}$ we have
\begin{eqnarray}\label{e3.q2}
\frac{(y_1-x_1)^2}{2^{-t}}
+\frac{(y_2-x_2)^2}{2^{-t}}=1\Longleftrightarrow |x-y|=2^{-\frac t2}.
\end{eqnarray}
By the definition of $\T_0$, we know this happens if
 $x,\,y$ satisfy \eqref{e3.q2} for some $t\le 0$ or for some $t>0$ and $|x_2|>2^{-t/3}$. Equivalently, we have either  $|x-y|\ge 1$ or $|x-y|<1$ and $|x_2|>|x-y|^{\frac23}$.
In either of two subcases,
$$\rh_{1}(x,\,y)=|\theta_{x,\,t}|=\pi|x-y|^2.$$

{\bf Case 2.} Suppose that $t>0$, $\sz_1=2^{-t/3}$, and $\sz_2=2^{-2t/3}$. Since $y\in\partial\theta_{x,\,t}$ we have
\begin{eqnarray}\label{e3.q3}
\frac{(y_1-x_1)^2}{2^{-\frac{2t}3}}
+\frac{(y_2-x_2)^2}{2^{-\frac{4t}3}}=1.
\end{eqnarray}
This is equivalent to $|x-y|<1$ and $|x_2|\le 2^{-\frac t2}$, where
$$2^{-\frac t2}=2^{-\frac34}\lf[(y_1-x_1)^2+\sqrt{(y_1-x_1)^4+4(y_2-x_2)^2}\r]^{\frac34}=:\varphi(x,\,y).$$
Therefore, in this case,
$$ \rh_{1}(x,\,y)=|\theta_{x,\,t}|=\pi 2^{-t}=\pi[\varphi(x,\,y)]^2.$$

{\bf Case 3.} Suppose that $t>0$, $\sz_1=2^{-5t/6}/|x_2|$, and $\sz_2=2^{-t/6}|x_2|$.
 Since $y\in\partial\theta_{x,\,t}$ we have
\begin{eqnarray*}
\frac{(y_1-x_1)^2}{[2^{-\frac {5t}6}|x_2|^{-1}]^2}
+\frac{(y_2-x_2)^2}{[2^{-\frac t6}|x_2|]^2}=1.
\end{eqnarray*}
Let $a=(y_1-x_1)^2|x_2|^2$, $b=(y_2-x_2)^2/|x_2|^2$, and $z=2^{t/3}$. Since Case 3 is complementary to Cases 1 and 2, we necessarily have  $|x-y|<1$, $\varphi(x,\,y)<|x_2|\le |x-y|^{2/3}$, and $az^5+bz=1$.

Since $a, b>0$ and $z>1$, by Lemma \ref{l3.ss1}, we have
$$ z=2^{\frac t3}\sim \min\{ (x_1-y_1)^{-\frac25}|x_2|^{-\frac 25},\, (x_2-y_2)^{-2}|x_2|^2     \}.  $$
Thus, we have
$$\rh_{1}(x,\,y)=|\theta_{x,\,t}|=\pi2^{-t}\sim \max\{(x_1-y_1)^{\frac65}|x_2|^{\frac65} ,\,(x_2-y_2)^6|x_2|^{-6}\}.$$

Combining Cases 1--3 with \eqref{e3.s1} shows Proposition \ref{p3.z2}.
\end{proof}

By Theorems \ref{auto} and \ref{t4.1} we deduce that $\rho_{\Theta_0}$ is quasi-convex and $1$-Ahlfors-regular quasi-distance which satisfies the inner property. However, these properties are far from obvious from the formula for $\rho_{\Theta_0}$ in Proposition \ref{p3.z2}

Next we will give an example of a quasi-convex quasi-distance $\rh$, which is not $1$-Ahlfors-regular, but which nevertheless generates a continuous ellipsoid cover. It will be convenient to relax the assumption of symmetry of quasi-distance in Definition \ref{d3.y1}  by the condition $\rh(x,\,y)\le C\rh(y,\,x)$ for any $x,\,y\in\R^n$, see \cite[Section I.2.4]{S}. This formally weaker condition implies that $\rh(x,\,y)\sim\rh(y,\,x)$. Hence, its symmetrization $[\rh(x,\,y)+\rh(y,\,x)]/2$ is a quasi-distance in the sense of Definition \ref{d3.y1}, albeit for (possibly) increased triangle constant $\kappa$.

The following example can be found in the monograph of Stein \cite[Section I.2.6]{S}. It is merely the simplest example of general class of balls and metrics studied by Nagel, Stein, and Wainger \cite{NSW}. Let $k$ be a non-negative integer and, for any $x\in\R^2$ and $\delta>0$, let
\begin{equation}\label{e3.w5}
B_k(x,\,\delta):=\{y\in\R^2:\, |x_1-y_1|<\delta,\,|x_2-y_2|<\max\{\delta^{k+1},\, |x_1|^k\delta\}\}.
\end{equation}
Then balls $\{B_k(x,\,\delta):\, x\in\R^2,\,\delta\in(0,\,\infty)\}$ are a natural family of balls associated with the vector fields $X_1:=\partial/\partial x_1$ and
$X_2:=x^k_1\partial/\partial x_2$. That is, $y\in B_k(x,\delta)$ if one can join $x$ to $y$ along a path whose velocity vector at any point is of the form $a_1X_1+a_2 X_2$, with $|a_1| \le 1$ and $|a_2|\le 1$, in elapsed time $\ls \delta$.
The balls $B_k(x,\,\delta)$ can be equivalently defined as $B_{\rh_k}(x,\,\delta):=\{y\in\R^2:\, \rh_k(y,\,x)<\delta\}$, where
\begin{equation}\label{e3.x1}
\rh_k(y,\,x)=
\lf\{
\begin{array}{ll}
\max\{|y_1-x_1|,\,\min\{|y_2-x_2|^{1/(k+1)},\,|y_2-x_2|/|x_1|^k\}\} & \mathrm{if}\ x_1\not= 0,\\
\max\{|y_1-x_1|,\, |y_2-x_2|^{1/(k+1)} \} & \mathrm{if}\ x_1=0.
\end{array}\r.
\end{equation}

\begin{prop}\label{p3.x5}
Let $k$ be a non-negative integer and let $\rh_k$ be as in \eqref{e3.x1}. Then $\rh_k$ is a quasi-distance which is quasi-convex and satisfies the inner property, but $\rho_k$ is not $1$-Ahlfors-regular.
\end{prop}

\begin{proof}
It is not difficult to check that $\rho_k$ is equivalent to the metric associated with vector fields $X_1$ and $X_2$, see \cite[Definition 1.1]{NSW}. That is, the distance between $x$ and $y$ is the infimum of travel times between $x$ and $y$ along paths whose velocity vector at any point is of the form $a_1X_1+a_2 X_2$, with $|a_1| \le 1$ and $|a_2|\le 1$.
By \eqref{e3.w5} any ball $B_{\rh_k}(x,\,\delta)$ is actually a rectangle. Hence, $\rho_k$ is quasi-convex with $Q=\sqrt{2}$.
Moreover, for any $\la\ge1,\,\delta>0$ and $x\in\R^2$, we have
\begin{align}\label{e5.1}
 \la \lf(B_{\rh_k}(x,\,\delta)-x\r)&=\lf\{ (y_1,\,y_2):\, |y_1|<\la \delta,\, |y_2|< \la\max\{\delta^{k+1},\,\delta|x_1|^k\}  \r\}\\
 &\subseteq \lf\{ (y_1,\,y_2):\, |y_1|<\la \delta,\, |y_2|<\max\{(\la\delta)^{k+1},\,\la\delta|x_1|^k\}  \r\}\nonumber\\
 &= B_{\rh_k}(x,\,\la\delta)-x.\nonumber
\end{align}
Hence, $\rh_k$ satisfies the inner property as in Definition \ref{d3.yy1} with $a=1$ and $b=1$.
By \eqref{e3.w5} we have
\begin{equation}\label{ml}
|B_{\rh_k}(x,\,\delta)| = 4 \delta^2 \max\{\delta^{k},|x_1|^k \}.
\end{equation}
Hence, the Lebesgue measure is not $1$-Ahlfors-regular with respect to $\rho_k$.
\end{proof}

In spite of Proposition \ref{p3.x5}, one can associate with $\rho_k$ a continuous ellipsoid cover. By the quasi-convexity of $\rho_k$ we can consider family of ellipsoids $\Xi_{\rho_k}=\{\xi_x^r: x\in \R^n, r>0\}$ as in Definition \ref{quasi convex}
\[
\xi^r_x\Sub B_{\rho_k}(x,r) \Sub Q \cdot \xi^r_x.
\]
 For any $x\in \R^n$ and $t\in \R$ define
\begin{equation}\label{mll}
\theta_{x,\,t}= \xi_x^{r(t)}, \qquad\text{where }r(t) =\sup\{r>0: |\xi^r_x| \le 2^{-t}\}.
\end{equation}
It follows from \eqref{e3.w5} and \eqref{mll} that $\Theta=\{\theta_{x,\,t}: x\in \R^n, t\in \R\}$ satisfies property (i) in Definition \ref{d3.x1}. It takes considerably more effort to show that $\Theta$ satisfies the shape condition (ii) using Lemma \ref{geo}. Consequently, the Hardy space $H^p(\R^n,\rho_k)$, which corresponds to ellipsoid cover $\Xi_{\rho_k}$, satisfies the conclusions of Theorem \ref{hardy}. We leave details to an interested reader.



\begin{thebibliography}{99}

\bibitem{abr}
V. Almeida, J. J. Betancor and L. Rodr\'{\i}guez-Mesa, {\it Molecules associated to Hardy spaces with pointwise variable anisotropy},  Integral Equations Operator Theory {\bf 89} (2017), 301--313.

\bibitem{am}
R. Alvarado, M. Mitrea,
{\it Hardy spaces on Ahlfors-regular quasi metric spaces.
A sharp theory.} Lecture Notes in Mathematics, 2142. Springer, Cham, 2015.

\bibitem{ball}
K. Ball,
{\it Ellipsoids of maximal volume in convex bodies},
Geom. Dedicata {\bf 41} (1992), no. 2, 241--250.


\bibitem{bll}
M. Bownik, B. Li, J. Li,
{\it Variable anisotropic singular integrals}, preprint {\tt arxiv:2004.09707v2}.


\bibitem{cw71}
R. R. Coifman and G. Weiss,
{\it Analyse Harmonique Non-commutative sur Certains Espaces Homog\`enes,}
Lecture Notes in Math. 242, Springer, Berlin, 1971.


\bibitem{cw77}
  R. R. Coifman and G. Weiss,
{\it Extensions of Hardy spaces and their use in analysis,}
Bull. Amer. Math. Soc. {\bf 83} (1977), no. 4, 569--645.

\bibitem{DDP0}
W. Dahmen, S. Dekel, P. Petrushev, 	
{\it Multilevel preconditioning for partition of unity
methods: some analytic concepts}, Numer. Math. {\bf 107} (2007), 503--532.

\bibitem{DDP}
W. Dahmen, S. Dekel, P. Petrushev, {\it Two-level-split decomposition of anisotropic Besov spaces}, Const. Approx. {\bf 31} (2010), 149--194.

\bibitem{De}
S. Dekel,
{\it On the analysis of anisotropic smoothness},
J. Approx. Theory {\bf 164} (2012), 1143--1164.


\bibitem{DHP}
S. Dekel, Y. Han, P. Petrushev,
{\it Anisotropic meshless frames on $\R^n$},
J. Fourier Anal. Appl. {\bf 15} (2009) 634--662.

\bibitem{DP}
S. Dekel, P. Petrushev,
{\it Anisotropic function spaces with applications},
Multiscale, Nonlinear and Adaptive Approximation, 137--167, Springer, Berlin, 2009.

\bibitem{DPW}
S. Dekel, P. Petrushev, T. Weissblat,
{\it Hardy spaces on $\R^n$ with pointwise variable anisotropy}, J. Fourier Anal. Appl. {\bf 17} (2011), 1066--1107.

\bibitem{dw}
S. Dekel and T. Weissblat, {\it On dual spaces of anisotropic Hardy spaces},
 Math. Nachr. {\bf 285} (2012), 2078--2092.




\bibitem{hs}
Y. S. Han, E. T. Sawyer,
{\it Littlewood-Paley theory on spaces of homogeneous type and the classical function spaces},
Mem. Amer. Math. Soc. {\bf 110} (1994), no. 530, vi+126 pp.

\bibitem{H}
J. Heinonen, {\it Lectures on analysis on metric spaces},
 Universitext. Springer-Verlag, New York, 2001.

\bibitem{FJ}
F. John, {\it Extremum Problems with Inequalities as Subsidiary Conditions},
Studies and Essays Presented to R. Courant on his 60th Birthday, January 8, 1948, Interscience Publishers, Inc., New York, N. Y., 1948, pp. 187--204.

\bibitem{NSW}
A. Nagel, E. M. Stein, S. Wainger,
{\it Balls and metrics defined by vector fields. I. Basic properties.}
Acta Math. {\bf 155} (1985), 103--147.

\bibitem{S}
E. M. Stein, {\it Harmonic Analysis: Real-Variable Methods, Orthogonality, and
 Oscillatory Integrals}. Princeton University Press, Princeton, N.J., 1993.

\bibitem{TV}
T. Tao, V. Vu, {\it Additive Combinatorics}, Cambridge Studies in Advanced
Mathematics, 105. Cambridge University Press, Cambridge, 2006.

\bibitem{YZ}
D. Yang, Y. Zhou,
{\it New properties of Besov and Triebel-Lizorkin spaces on RD-spaces},
Manuscripta Math. {\bf 134} (2011), no. 1--2, 59--90.

\end{thebibliography}
\end{document}